\documentclass[12pt]{article}
\usepackage{xcolor}% http://ctan.org/pkg/xcolor
\usepackage{algpseudocode}% http://ctan.org/pkg/algorithmicx
\usepackage{algorithm}% http://ctan.org/pkg/algorithm
\usepackage{amssymb}
\usepackage{multicol}
\usepackage{mathtools}
\usepackage{graphicx}
\usepackage[mathscr]{euscript}
\usepackage{amsthm}
\usepackage{verbatim}
\usepackage{algorithmicx}
\usepackage{algorithm}
\usepackage{algpseudocode}
\usepackage{fixltx2e}
\usepackage{hyperref}
\usepackage{amsmath, epsfig}
\usepackage{calc}
\usepackage{tikz}
\usepackage{color}
\usepackage{caption}
\allowdisplaybreaks
\usetikzlibrary{decorations.markings}
\tikzstyle{vertex}=[circle, draw, inner sep=0pt, minimum size=6pt]
\usepackage{tikz-cd}
\usetikzlibrary{matrix,arrows,decorations.pathmorphing}

\usepackage{geometry}
\newcommand{\vertex}{\node[vertex]}

\newtheorem{theorem}{Theorem}
\newtheorem{lemma}{Lemma}
\newtheorem{conjecture}{Conjecture}
\newtheorem{proposition}{Proposition}  
\newtheorem{remark}{Remark}  
\newtheorem{corollary}{Corollary}
\newtheorem{definition}{Definition}

\newtheorem{claim}{Claim}

\newcommand{\vol}{\mathrm{vol}}

\title{The maximum relaxation time of a random walk}

\author{ Sinan G. Aksoy \thanks{Pacific Northwest National Laboratory, Richland, WA,
({\tt sinanaksoy90@gmail.com}).}
\and Fan Chung \thanks{Department of Mathematics, University of California, San Diego,
({\tt fan@ucsd.edu}).}
\and Michael Tait \thanks{Department of Mathematical Sciences, Carnegie Mellon University, ({\tt mtait@cmu.edu}). Research is partially supported by NSF grant DMS-1606350.}
\and Josh Tobin \thanks{Department of Mathematics, University of California, San Diego
({\tt rjtobin@ucsd.edu}).}
}

\date{}
\begin{document}
\maketitle
\abstract{We show the minimum spectral gap of the normalized Laplacian over all simple, connected graphs on $n$ vertices is $(1+o(1))\tfrac{54}{n^3}$. This minimum is achieved asymptotically by a double kite graph. Consequently, this leads to sharp upper bounds for the maximum relaxation time of a random walk,  settling  a conjecture of Aldous and Fill. We also improve an eigenvalue-diameter inequality by giving a new lower bound for the spectral gap  of the normalized Laplacian. This eigenvalue lower bound is asymptotically best possible. \\

}

\section{Introduction}

Graph eigenvalues play a powerful role in the study of random walks. In particular, eigenvalues are a primary tool for bounding a number of key random walk parameters, such as mixing time. Consequently, bounds on graph eigenvalues are not only of interest in themselves, but also may have immediate implications for the behavior of the random walk (for a survey, see \cite{lovasz1993random}). In the case of the {\it relaxation time} of a discrete reversible Markov chain, eigenvalues themselves define the quantity of interest. 

In this paper, we examine an extremal problem concerning the normalized Laplacian spectral gap, the reciprocal of which defines the relaxation time of a random walk. The normalized Laplacian matrix $\mathcal{L}$ of a graph $G$ is 
\begin{align*}
\mathcal{L}&= I-T^{-1/2}AT^{-1/2},
\end{align*}
where $T$ denotes the diagonal degree matrix with $(u,u)$ entry equal to $d(u)$ and $A$ denotes the adjacency matrix. Throughout, we assume $G$ is simple, meaning $G$ has no loops or multiple edges. We write the eigenvalues of $\mathcal{L}$ in increasing order, where
\[
0=\lambda_0\leq\lambda_1\leq\dots\leq\lambda_{n-1} \leq 2.
\]
It is well-known (c.f. \cite{chung1997spectral}) that the second eigenvalue or spectral gap of $\mathcal{L}$ is nonzero if and only if $G$ is connected, and can be characterized as
\[
\lambda_1 =  \inf_{\substack{f \\ \sum_{u}f(u)d(u)=0}}\frac{\displaystyle\sum_{u \sim v} (f(u)-f(v))^2}{\displaystyle\sum_{v}f(v)^2 d(v)},
\]
with corresponding eigenvector $g=T^{1/2}f$.  We call the nontrivial function $f$ achieving the above infimum the {\it harmonic eigenfunction} of $\mathcal{L}$. Landau and Odlyzko proved the following lower bound on $\lambda_1$.

\begin{theorem}[Landau, Odlyzko \cite{Land81}] \label{thm:land}
For a connected graph on $n$ vertices with maximum degree $\Delta$ and diameter $D$, we have
\[
\lambda_1 \geq \frac{1}{n \Delta (D+1)}.
\]
\end{theorem}

In \cite{chung1997spectral}, Chung gives an improved lower bound on $\lambda_1$ in terms of the graph's diameter and volume, where $\vol(G)= \sum_{u \in V(G)} d(u)$.
\begin{theorem}[Chung \cite{chung1997spectral}] \label{fansBound}
For a connected graph $G$ with diameter $D$, we have
\[
\lambda_1 \geq \displaystyle \frac{1}{D \cdot \vol (G)}.
\]
\end{theorem}

For symmetrical graphs, stronger lower bounds may be obtained. For example, Chung showed that for a vertex-transitive graph with degree $k$ and diameter $D$, we have
\[
\lambda_1 \geq \frac{1}{kD^2}.
\]

In this paper, we have two main results. First, we improve the constant in the statement of Theorem \ref{fansBound}.

\begin{theorem}\label{fanImprovement}
For a connected graph $G$ with diameter $D$, we have
\[
\lambda_1 \geq \frac{4}{D \cdot \vol({G})}.
\]
\end{theorem}
The above lower bound is in fact asymptotically best possible (see further discussions later in Remark \ref{rem:sharp}). Second, we examine the minimal value of $\lambda_1$ over all connected graphs on $n$ vertices.

\begin{theorem} \label{min54}
The minimum normalized Laplacian spectral gap $\alpha(n)$, defined by
\[
 \alpha(n)=\min\{\lambda_1(G): G \mbox{ is a simple, connected graph on $n$ vertices} \}
\]
satisfies
\[
\alpha(n) \sim \frac{54}{n^3}.
\]
\end{theorem}

As an immediate consequence of Theorem \ref{min54}, we confirm a conjecture of Aldous and Fill on {\it relaxation time}. The {relaxation time} $\tau$ of a random walk on a (connected) graph $G$ with probability transition matrix $P=T^{-1}A$ is defined as
\[
\tau(G) = \frac{1}{1-\rho_{n-1}},
\]
where $\rho_1\leq \dots \leq \rho_{n-1} < \rho_n=1$ denote the eigenvalues of $P$. A central problem in the study of random walks is to determine the {\it mixing time}, the required number of steps in the random walk guaranteeing closeness to the stationary distribution. As seen throughout the literature \cite{aldous2002reversible, chung1997spectral, levin2017markov}, the eigenvalue $\rho_{n-1}$ and hence the relaxation time is the primary term controlling mixing time. Therefore, relaxation time is directly associated with the rate of convergence for a random walk. 
At least as early as 1994, Aldous and Fill \cite[Problem 6.13, p.~216]{aldous2002reversible} conjectured the following concerning relaxation time:
\begin{conjecture}[Aldous and Fill, c.~1994] \label{conj:af} 
The maximum relaxation time $\beta(n)$, defined by
\[
\beta(n)=\max\{\tau(G): G \mbox{ is a simple, connected graph on $n$ vertices} \},
\]
satisfies
\[
\beta(n) \sim \frac{n^3}{54}.
\]
\end{conjecture}
In \cite{aldous2002reversible}, Aldous and Fill showed that $\beta(n)$ is bounded above by $(1+o(1))\frac{2n^3}{27}$. In general, Conjecture \ref{conj:af} fits into a body of work addressing extremal problems for random walk parameters. For example, Brightwell and Winkler \cite{brightwell1990maximum} found the maximum hitting time between two vertices over all $n$-vertex graphs and determined the extremal graphs are lollipop graphs. Relatedly, Mazo considered maximum and minimum mean hitting time \cite{mazo1982some}. Furthermore, Feige obtained sharp upper bounds on cover time  \cite{feige1995tight, feige1996collecting}, and Coppersmith, Tetali, and Winkler found the maximum commute time \cite{coppersmith1993collisions}.

It is easy to see that $T^{-1/2}\mathcal{L} T^{1/2}=I-T^{-1}A$, and hence $\lambda_i$ is an eigenvalue of $\mathcal{L}$ if and only if $1-\rho_{i}$ is an eigenvalue of $T^{-1}A$. Consequently, the relaxation time of a graph may equivalently be written as $\tau=1/{\lambda_1}$ and so Theorem \ref{min54} confirms Conjecture \ref{conj:af}. 

\begin{corollary}
The maximum relaxation time $\beta(n)$ for the random walk on a simple, connected graph on $n$ vertices satisfies $\beta(n) \sim n^3/54$. The extremal value $\beta(n)$ is achieved asymptotically by a double kite graph, $DK(\frac{n}{3}, \frac{n}{3})$.
\end{corollary}

The double kite graph can be defined as follows:

\begin{definition}
A {\it double kite graph}, denoted $DK(r,s)$, consists of two copies of the $r$-vertex complete graph $K_r$ and a path connecting them, $p_0,p_1,\dots,p_s,p_{s+1}$, where $p_0$ is a selected vertex from one copy of $K_r$ and $p_{s+1}$ is a selected vertex from the other copy of $K_r$. See Figure $\ref{kitePic}$ for an illustration. 
\end{definition}

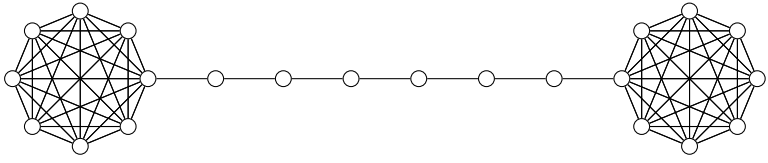
\begin{figure}[h]
\[\begin{tikzpicture}[scale=0.9]
	\vertex (c1) at (45:1){};
	\vertex (c2) at (90:1){};
	\vertex (c3) at (135:1){};
	\vertex (c4) at (180:1){};
	\vertex(c5) at (225:1){};
	\vertex(c6) at (270:1){};
	\vertex(c7) at (315:1){};
	\vertex(c8) at (360:1){};
	\vertex (d2) at (3,0){};
	\vertex (d3) at (4,0){};
	\vertex (d1) at (2,0){};
	\vertex (d4) at (5,0){};
	\vertex (d5) at (6,0){};
	\vertex (d6) at (7,0){};
	\vertex (e1) at (8,0){};
	\vertex (e2) at (9.7071, 0.7071){};	
	\vertex (e3) at (9,1){};
	\vertex (e4) at (8.2929,0.7071){};
	\vertex (e5) at (8.2929,-0.7071){};
	\vertex (e6) at (9,-1){};
	\vertex (e7) at (9.7071,-0.7071){};
	\vertex (e8) at (10,0){};
	\path 
		(c1) edge (c2) (c1) edge (c3) (c1) edge (c4) (c1) edge (c5) (c1) edge (c6) (c1) edge (c7)
		(c1) edge (c8) (c2) edge (c1) (c2) edge (c3) (c2) edge (c4) (c2) edge (c5) (c2) edge (c6)
		(c2) edge (c7) (c2) edge (c8) (c3) edge (c2) (c3) edge (c1) (c3) edge (c4) (c3) edge (c5)
		(c3) edge (c6) (c3) edge (c7) (c3) edge (c8) (c4) edge (c2) (c4) edge (c3) (c4) edge (c1)
		(c4) edge (c5) (c4) edge (c6) (c4) edge (c7) (c4) edge (c8) (c5) edge (c2) (c5) edge (c3)
		(c5) edge (c4) (c5) edge (c1) (c5) edge (c6) (c5) edge (c7) (c5) edge (c8) (c6) edge (c2)
		(c6) edge (c3) (c6) edge (c4) (c6) edge (c5) (c6) edge (c1) (c6) edge (c7) (c6) edge (c8) 
		(c7) edge (c2) (c7) edge (c3) (c7) edge (c4) (c7) edge (c5) (c7) edge (c6) (c7) edge (c1)
		(c7) edge (c8) (c8) edge (c2) (c8) edge (c3) (c8) edge (c4) (c8) edge (c5) (c8) edge (c6)
		(c8) edge (c1) (c8) edge (c7) (c8) edge (d1)
		
		(d1) edge (d2) (d2) edge (d3) (d3) edge (d4) (d4) edge (d5) (d5) edge (d6) 
		
		(e1) edge (e2) (e1) edge (e3) (e1) edge (e4) (e1) edge (e5) (e1) edge (e6) (e1) edge (e7)
		(e1) edge (e8) (e2) edge (e1) (e2) edge (e3) (e2) edge (e4) (e2) edge (e5) (e2) edge (e6)
		(e2) edge (e7) (e2) edge (e8) (e3) edge (e2) (e3) edge (e1) (e3) edge (e4) (e3) edge (e5)
		(e3) edge (e6) (e3) edge (e7) (e3) edge (e8) (e4) edge (e2) (e4) edge (e3) (e4) edge (e1)
		(e4) edge (e5) (e4) edge (e6) (e4) edge (e7) (e4) edge (e8) (e5) edge (e2) (e5) edge (e3)
		(e5) edge (e4) (e5) edge (e1) (e5) edge (e6) (e5) edge (e7) (e5) edge (e8) (e6) edge (e2)
		(e6) edge (e3) (e6) edge (e4) (e6) edge (e5) (e6) edge (e1) (e6) edge (e7) (e6) edge (e8) 
		(e7) edge (e2) (e7) edge (e3) (e7) edge (e4) (e7) edge (e5) (e7) edge (e6) (e7) edge (e1)
		(e7) edge (e8) (e8) edge (e2) (e8) edge (e3) (e8) edge (e4) (e8) edge (e5) (e8) edge (e6)
		(e8) edge (e1) (e8) edge (e7) (d6) edge (e1)
		
	;
\end{tikzpicture}
\]
\caption{The double kite graph $DK(8,6)$.} \label{kitePic}
\end{figure}

\begin{remark} 
In \cite{aldous2002reversible}, Aldous and Fill call $DK(r,s)$ the {\it barbell graph}. The specific cases of $DK(\frac{n}{2},0)$ as well as $DK(\frac{n}{3},\frac{n}{3})$ have also both been commonly referred to as the barbell graph (e.g., see \cite{ghosh2008minimizing} and \cite{wilf1989editor} respectively). 
\end{remark}

\begin{remark}
Landau and Odlyzko also consider the construction $DK(\frac{n}{3},\frac{n}{3})$ to show that the $n^3$ order of magnitude implied by their bound (Theorem \ref{thm:land}) is best possible. Applying their bound to this construction yields $\lambda_1 \geq (1+o(1))\frac{9}{n^3}$, while we show, $\lambda_1 \sim \frac{54}{n^3}$.
\end{remark}

\begin{remark} \label{rem:sharp}
We note that the bound in Theorem \ref{fanImprovement} is asymptotically tight for $DK(\frac{n}{3},\frac{n}{3})$, yielding $\lambda_1 \geq (1+o(1))\frac{54}{n^3}$. In general, however, the lower bound $4/D\cdot \vol(G)$ may be off by orders of magnitude. For example, applying the bound to the $d$-dimensional hypercube graph on $n=2^d$ vertices yields $\lambda_1 \geq \tfrac{4}{n \cdot \log_2^2(n)}$ yet $\lambda_1=\tfrac{2}{\log_2(n)}$.
%For example, applying the bound to the complete graph $K_n$ yields $\lambda_1 \geq \frac{4}{n(n-1)}$, whereas $\lambda_1=\frac{n}{n-1}$. Furthermore, even in cases where $\lambda_1=o(1)$, it may be that $4/D \cdot \vol(G)=o(\lambda_1)$, as evidenced by the hypercube graph, $Q_n$. 
On the other hand, in Section \ref{sec:fanImprov} we show Theorem \ref{fanImprovement} is sharp in a strong sense: for a wide range of $D$ and $\vol(G)$ there is an infinite sequence of graphs for which it is tight asymptotically, including the multiplicative constant.
\end{remark}

In addition to its interpretation in the random walk setting, Theorem \ref{min54} is also part of the literature surrounding extremal spectral graph theory, where one optimizes a spectral invariant over a fixed family of graphs. Such problems were first formalized by Brualdi and Solheid \cite{brualdi1986spectral} and since then have attracted attention from many researchers. Rather than give a broad survey of such work, we briefly mention a few results directly relevant to ours. For the spectral gap of the adjacency matrix, Stanic \cite{stanic2013graphs} proved some lower bounds for the spectral gap of the adjacency matrix, and conjectured that double kite graphs minimize the adjacency spectral gap. For the combinatorial Laplacian, Fallat and Kirkland \cite{fallat1998extremizing} find the combinatorial Laplacian algebraic connectivity minimizing graphs over all $n$-vertex trees with given diameter. Brand, Guiduli, and Imrich \cite{brand2007characterization} minimized $\lambda_1$ of the Laplacian over all $3$-regular graphs, and characterized the extremal graphs. For the general case, \cite{biyikouglu2012graphs} showed that the $n$-vertex graphs minimizing algebraic connectivity must consist of a chain of cliques. 

The remainder of the paper is structured as follows: in Section \ref{sec:fanImprov}, we prove a lemma from which Theorem \ref{fanImprovement} follows as a corollary and show Theorem \ref{fanImprovement} is sharp for a wide range of values of $D$ and $\mathrm{vol}(G)$. In Section \ref{sec:mainThm}, we apply this lemma, among others, to also prove Theorem \ref{min54}. In Section \ref{sec:conc}, we conclude by mentioning related open problems.

\section{Proof of Theorem \ref{fanImprovement}} \label{sec:fanImprov}

In this section, we establish the lemma from which Theorem \ref{fanImprovement} will follow as a corollary. To establish this lemma, we first require the solution to a related optimization problem.

\begin{proposition}\label{optimization}
Fix $(d_1,\ldots, d_n) \in \mathbb{N}^n$. Let $(f_1,\ldots, f_n)$ be a sequence minimizing the quantity
\[
(f_n-f_1)^2
\]
subject to the constraints
\begin{align}
\sum_{i=1}^n f_id_i = 0, \label{const:1}\\
\sum_{i=1}^n f_i^2d_i = 1\label{const:2},
\end{align}
and 
\[
f_1 \leq f_k \leq f_n
\]
for all $k$.
Then for all $k$ either $f_1 = f_k$ or $f_n = f_k$.
\end{proposition}
\begin{proof}
First we consider the optimization problem without the constraint that $f(1) \leq f(k) \leq f(n)$. In this case, consider the Lagrangian 
\[
(f_n-f_1)^2-\alpha\left(\sum_{i=1}^n f_id_i \right) - \beta \left(\sum_{i=1}^n f_i^2 d_i-1\right).
\]
We show that either we are on the boundary where there exists a $k$ such that $f(1) = f(k)$ or $f(n) = f(k)$, or the critical point of this Lagrangian maximizes the objective function $(f_n - f_1)^2$, and so the minimum must occur on the boundary. A critical point of the Lagrangian occurs when
\begin{align}
2\left(f_n-f_1\right)-\alpha d_n - 2\beta f_n d_n &= 0 \label{lag:1} \\
-2(f_n-f_1)-\alpha d_1 - 2\beta f_1 d_1 &= 0 \label{lag:2}\\
\alpha d_i + 2 \beta f_i d_i &= 0, \label{lag:3}
\end{align}
for $i=2,\dots, n-1$. If $\beta=0$, then from Eq.~$(\ref{lag:3})$, $\alpha=0$, in which case subtracting Eq.~$(\ref{lag:1})$ from Eq.~$(\ref{lag:2})$ yields $f_1=f_n$. But from the definitions of $f$ and $d$ and Eq.~$(\ref{const:1})$, it is clear $f_n>0$ and $f_1 <0$. So $\beta \not = 0$ and $f_i=-\frac{\alpha}{2 \beta}$ for $i=2,\dots,n-1$. Applying this fact and rewriting Eqs.~$(\ref{const:1})$ and Eq.~$(\ref{const:2})$ yields
\begin{align}
f_1d_1+f_nd_n &= \frac{\alpha}{2\beta}\sum_{i=2}^{n-1} d_i, \label{eq:firstlast} \\
f_1^2 d_1 + f_n^2 d_n &= 1- \frac{\alpha^2}{4 \beta^2} \sum_{i=2}^{n-1} d_i. \label{eq:square}
\end{align} 

Adding Eqs.~$(\ref{lag:1})$ and $(\ref{lag:2})$, then applying Eq.~$(\ref{eq:firstlast})$ yields
\[
\alpha \sum_{i=1}^n d_i =0,
\]
from which we can see that $\alpha=0$. Now, Eqs.~$(\ref{lag:3}),(\ref{eq:firstlast}),(\ref{eq:square})$ tell us $f_i =0$  for $i=2,\dots,n-2$, and
\begin{align*}
f_1 d_1 + f_n d_n &= 0, \\
f_1^2 d_1 + f_n^2 d_n &=1. 
\end{align*}

Rewriting the former equation above, we get $f_1 = -c \cdot d_n$ and $f_n=c \cdot d_1$ for $ c\coloneqq {f_n}/{d_1}$. Plugging this into the latter, we find
\[
c^2=\frac{1}{d_1d_n(d_1+d_n)}.
\]
Finally, we have
\[
(f_n-f_1)^2=c^2 (d_1+d_n)^2 = \frac{1}{d_1}+\frac{1}{d_n}.
\]
We claim that this is the maximum value of $(f_n-f_1)^2$ subject to the constraints. To see this, note that letting 
\[
f_1 = f_2 = -\sqrt{\frac{d_n}{(d_1+d_2)(d_1+d_2+d_n)}}, \quad \quad f_n = \sqrt{\frac{d_1+d_2}{d_n(d_1+d_2+d_n)}},
\]
satisfies all of the constraints and gives 
\[
(f_n-f_1)^2 = \frac{1}{d_1+d_2} + \frac{1}{d_n},
\]
which is smaller than $\frac{1}{d_1} + \frac{1}{d_n}$ since $d_2\geq 1$. Therefore, the only critical point of the Lagrangian interior to the boundary is a maximum, and thus the minimum must occur when there is a $k$ such that $f_1 = f_k$ or $f_n = f_k$. In this case, we may substitute for $f_k$, and we are left with a similar optimization problem in $n-1$ variables, where we have eliminated the variable $f_k$ and replaced $d_1$ with $d_1+d_k$ if $f_1=f_k$ or $d_n$ by $d_n + d_k$ if $f_n = f_k$. We may use this argument repeatedly to show that the minimum must occur on the boundary until there are only $2$ variables remaining. At this point, the objective function is constant subject to the constraints, and we are done.

\end{proof}

We now prove the lemma from which Theorem \ref{fanImprovement} will follow. Let $G$ be a connected graph with normalized Laplacian eigenvalues $\lambda_0 \leq \lambda_1 \leq  \cdots \leq \lambda_{n-1}$, and let $f$ be a harmonic eigenvector for $\lambda_1$. Once $f$ is fixed, let $u$ and $v$ be vertices corresponding to minimum and maximum entries of $f$ respectively. That is, for all $z\in V(G)$ we have $f(u) \leq f(z) \leq f(v)$. Further, let 
\begin{align*}
\vol_P &= \sum_{z: f(z) \geq 0} d(z),\\
\vol_N &= \sum_{z: f(z) < 0} d(z).
\end{align*}

\begin{lemma}\label{lowerBoundLemma}
Let $G$ be a connected graph with $f$ a harmonic eigenvector for $\lambda_1$ of its normalized Laplacian. Let $u$ and $v$ be vertices which minimize and maximize $f$ respectively, and let $\vol_P$ and $\vol_N$ be defined as above. Then 
\[
\lambda_1 \geq \frac{2}{\mathrm{dist}(u,v) \sqrt{\vol_P \cdot \vol_N}}.
\]
\end{lemma}
\begin{proof}
Let $f$ be a harmonic eigenvector for $\lambda_1$, and let $u$ and $v$ be vertices which minimize and maximize $f$ respectively, so $f(u) \leq f(z) \leq f(v)$ for all $z\in V(G)$. Let $S$ be a shortest path from $u$ to $v$. Then, 
\begin{align*}
\lambda_1 &= \frac{\sum_{x\sim y} (f(x)-f(y))^2}{\sum_x (f(x))^2d(x)}\\
&\geq \frac{\sum_{xy\in S} (f(x)-f(y))^2}{\sum_x (f(x))^2d(x)} \\
& \geq \frac{\frac{1}{|S|} (f(u) - f(v))^2}{\sum_x (f(x))^2d(x)},
\end{align*}
where the last inequality is by Cauchy-Schwarz. Now, since $f$ is a harmonic eigenvector, we have 
\[
\sum_x f(x)d(x) = 0.
\]
We may without loss of generality scale $f$ so that 
\[
\sum_x (f(x))^2d(x) = 1.
\]
By Proposition \ref{optimization}, we have that the quantity $(f(u)-f(v))^2$ is bounded below by $(c_2-c_1)^2$ where $c_1$ and $c_2$ satisfy
\[
\sum_{x\in N} c_1d(x) + \sum_{x\in P} c_2d(x) = 0,
\]
and
\[
\sum_{x\in N} c_1^2 d(x) + \sum_{x\in P} c_2^2 d(x) = 1.
\]
If $c_1$ and $c_2$ satisfy this system, then we have 
\[
c_1 = - \sqrt{\frac{\vol_P}{\vol_N^2 + \vol_P\vol_N}}, \quad \quad c_2 = \sqrt{\frac{\vol_N}{\vol_P^2 + \vol_P\vol_N}}.
\]
Thus we have 
\[
\lambda_1 \geq \frac{1}{\mathrm{dist}(u,v)} \left(\sqrt{\frac{\vol_N}{\vol_P^2 + \vol_P \vol_N}} + \sqrt{\frac{\vol_P}{\vol_N^2 + \vol_P \vol_N}}\right)^2.
\]
Using calculus, one can see that 
\[
 \left(\sqrt{\frac{\vol_N}{\vol_P^2 + \vol_P\vol_N}} + \sqrt{\frac{\vol_P}{\vol_N^2 + \vol_P\vol_N}}\right)^2 \geq \frac{2}{\sqrt{\vol_P\vol_N}}.
\]

\end{proof}

As a corollary of this, we can now prove Theorem \ref{fanImprovement}.
\begin{proof}[Proof of Theorem \ref{fanImprovement}]
Note that $\vol(G) = \vol_P + \vol_N$, and so the AM-GM inequality gives us 
\[
\frac{\vol(G)}{2} \geq \sqrt{\vol_P\cdot \vol_N}.
\]
Now, if $D$ is the diameter of $G$, we have by Lemma \ref{lowerBoundLemma} that 
\[
\lambda_1 \geq \frac{2}{\mathrm{dist}(u,v) \sqrt{\vol_P\vol_N}} \geq \frac{2}{D \sqrt{\vol_P\vol_N}} \geq \frac{4}{D\cdot \vol(G)}.
\]
\end{proof}

Next we give a family of constructions showing that Theorem \ref{fanImprovement} is sharp. \begin{proposition}\label{prop:construction}
Let $D$ and $d$ be fixed, and let $n-D+1$ be divisible by $4$. Let $H_1$ and $H_2$ be $d$-regular graphs on $\frac{n-D+1}{2}$ vertices, and let $H$ be the graph obtained by joining $H_1$ and $H_2$ by a path of length $D$. Then 
\[
\lambda_1(H) \leq \frac{4}{Dd(n-D)}.
\]
\end{proposition}
\begin{proof}
Label the vertices on the path between $H_1$ and $H_2$ as $p_0, p_1,\ldots, p_D$, where the terminal vertices $p_0$ and $p_D$ belong to $H_1$ and $H_2$ respectively. Define $f: V(H) \to \mathbb{R}$ by 
\[
f(u) =
\begin{cases}
1 & \mbox{if } u\in H_1, \\
-1 & \mbox{if } u\in H_2, \\
1-\frac{2i}{D} &\mbox{if } u =p_i.
\end{cases}
\]

One may check that $\sum_{u} f(u)d(u) = 0$, and hence 
\begin{align*}
\lambda_1 \leq \frac{\sum_{u \sim v} (f(u)-f(v))^2}{\sum_{v}f(v)^2 d(v)}  \leq \frac{\sum_{u \sim v} (f(u)-f(v))^2}{(n-D)d} &= \frac{\sum_{i=1}^D (f(p_i) - f(p_{i-1}))^2}{(n-D)d} \\ &= \frac{D\left(\frac{2}{D}\right)^2}{(n-D)d}. \end{align*}
\end{proof}

Now, given $H$ we have that $\mathrm{vol}(H) = (n-D+1)d + 2D$ and the diameter of $H$ is at most $D +\mathrm{diam}(H_1) + \mathrm{diam(H_2)}$. Therefore, as long as we have $d(n-D+1)+2D \sim d(n-D)$ and $\mathrm{diam}(H_1) + \mathrm{diam}(H_2) = o(D)$, then the lower bound in Theorem \ref{fanImprovement} is asymptotically tight for $\lambda_1(H)$ as $n$ goes to infinity. Since we may choose $d$-regular graphs with diameter $O(\log n)$, for any $D$ and $V$ satisfying $D \gg \log n$ and $n\ll V \leq \frac{n^2}{2}$, there is a sequence of graphs with diameter asymptotic to $D$ and volume asymptotic to $V$ for which the bound in Theorem \ref{fanImprovement} is asymptotically sharp.

\section{Proof of Theorem \ref{min54}} \label{sec:mainThm}

We first prove an upper bound on $\alpha(n)$, which is straightforward by considering the double kite graph.

\begin{claim} \label{upperBound}
\[
\alpha(n) \leq (1+o(1))\frac{54}{n^3}.
\]
\end{claim}

\begin{proof}
Consider $G=DK(\frac{n}{3},\frac{n}{3})$. By Proposition \ref{prop:construction} we have $\lambda_1(G) \leq (1+o(1))\frac{54}{n^3}$.
\end{proof}

It remains to prove that $\alpha(n) \geq (1+o(1))\frac{54}{n^3}$. To do so, we will use Lemma \ref{lowerBoundLemma} from Section \ref{sec:fanImprov}, as well as an additional lemma below that establishes a key property of the extremal graphs. Henceforth, we assume $G$ achieves $\alpha(n)$ with harmonic eigenvector $f$ satisfying
\[
\lambda_1 = \frac{\sum_{x\sim y}(f(x)-f(y))^2}{\sum_x (f(x))^2d(x)}.
\]
Let 
\begin{align*}
P &= \{z\in V(G): f(z) \geq 0\},\\
N &=\{z\in V(G): f(z) < 0\}.
\end{align*}
Further, let $u$ and $v$ satisfy $f(u) \leq f(z) \leq f(v)$ for all $z\in V(G)$ and let $S$ be a shortest path from $u$ to $v$.

\begin{lemma}\label{NPedges}
If $G$ achieves $\alpha(n)$, then the number of edges with one endpoint in $N$ and the other in $P$ satisfies 
\[
1\leq e(N,P) \leq n-1.
\]
\end{lemma}

\begin{proof}
Since $f$ is a harmonic eigenvector, we have $\sum_x f(x)d(x) = 0$ and so $f(u) < 0 < f(v)$. Therefore, there must be an edge in $S$ that has one endpoint in $N$ and the other in $P$. To see the upper bound, we claim that any edge with one endpoint in $N$ and the other in $S$ must be a bridge. To see this, let 
\[
a = \sum_{x\sim y} (f(x) - f(y))^2,
\]
and 
\[
b = \sum_x (f(x))^2d(x),
\]
so that $\lambda_1 = \frac{a}{b}$. Now let $e=wz$ be an edge with one endpoint in $N$ and the other in $P$, and let $G' = G\setminus \{e\}$. Furthermore, let $d'(x)$ be the degree sequence of $G'$, and let $f'(x) = f(x) + c$ where $c$ is chosen so that $\sum_x f'(x) d'(x) = 0$.  So
\begin{eqnarray*}
  0 = \sum_x \left( f(x) + c \right) d'(x) & = & \sum_x \left( f(x) + c \right) d(x) - f(w) - c - f(z) - c \\
  & = & \sum_x f(x) d(x) + c \sum_x d(x) - f(z) - f(w) - 2c \\
  & = & c \sum_x d(x) - 2c - f(z) - f(w).
\end{eqnarray*}
We get
\begin{equation}\label{c_expression}
c = \frac{f(z) + f(w)}{\sum_x d(x) - 2}.
\end{equation}

If $R_G(f)$ is the Rayleigh quotient of graph $G$ with harmonic eigenfunction $f$, then define $c_1, c_2$ so that
\[ R_{G'}(f') = \frac{a -c_1 }{b - c_2}, \]
where $c_1, c_2 > 0$.  It is easily seen that
\[ \frac{a-c_1}{b-c_2} < \frac{a}{b} \]
if and only if
\[ \lambda_1 = \frac{a}{b} < \frac{c_1}{c_2}.\]

By definition of $f'$ and $G'$, we have $c_1 = (f(w) - f(z))^2 > f(w)^2 + f(z)^2$, since $f(w)f(z) < 0$.  Also,
\begin{eqnarray*}
  c_2 & = & \sum_x f(x)^2 d(x) - \sum_x f'(x)^2d'(x) \\
  & = & \sum_x f(x)^2 d(x) - \left(\sum_x (f(x)+c)^2 d(x) - (f(z) + c)^2 - (f(w) + c)^2 \right) \\
  & = & f(z)^2 + f(w)^2 + 2c (f(z) + f(w)) - c^2 \left(\sum_x d(x) - 2\right).
\end{eqnarray*}
Using Expression~\ref{c_expression} we get
\[ c_2 = f(z)^2 + f(w)^2 + \frac{(f(z) + f(w))^2}{\sum_x d(x) - 2} \leq f(z)^2 + f(w)^2 + \frac{f(z)^2 + f(w)^2}{\sum_x d(x) - 2}, \]
again using the fact that $f(w)f(z) < 0$.  Combining these, we get
\[
\frac{c_1}{c_2} > \frac{f(z)^2 + f(w)^2}{f(z)^2 + f(w)^2 + \frac{f(z)^2 + f(w)^2}{\sum_x d(x) - 2}} = \frac{1}{1 + (\sum_x d(x) - 2)^{-1}}.
\]
If $G'$ is connected, we have the (very weak) bound $\sum_x d(x) - 2 > 2n - 4$, so for any $\varepsilon > 0$ if $n$ is large enough we have
$\frac{c_1}{c_2} > 1 - \varepsilon > \lambda_1$.  Therefore deleting this edge would decrease $\lambda_1$.  By minimality we conclude that
$e$ is a bridge. Now, given a connected graph, take any connected spanning tree. Since any edge not on this spanning tree cannot disconnect the graph, there can be at most $n-1$ bridges, giving us the upper bound.

\end{proof}

We are now in a position to prove a lower bound on $\alpha(n)$,  which completes our proof of Theorem \ref{min54}.

\begin{claim}
\[
\alpha(n)\geq (1+o(1))\frac{54}{n^3}.
\]
\end{claim}
\begin{proof}
Assume $G$ achieves $\alpha(n)$. Let $P' = P\setminus S$ and $N' = N\setminus S$, and let $|P'| = \alpha_1 n$, $|N'| = \alpha_2 n$, and $|S| = \alpha_3 n$. So $\alpha_1 + \alpha_2 + \alpha_3 = 1$. Now, since $S$ is a shortest path from $u$ to $v$, we have that any vertex in $V(G) \setminus S$ may have at most $3$ neighbors on $S$, and any vertex in $S$ may have at most $2$ neighbors in $S$. Letting $G_P$ and $G_N$ be the graphs induced by $P$ and $N$, respectively, note that 
\[
\vol_P = 2e(G_P) + e(N,P),
\]
and 
\[
\vol_N = 2e(G_N) + e(N,P).
\]
Putting these facts together, we have
\[
2e(G_P) \leq \sum_{z\in P} d(z) \leq |P'|^2 + 2e(P',S) + 2|S| \leq |P'|^2 + 6|P'| + 2|S| \leq \alpha_1^2n^2 + 8n.
\]
By Lemma \ref{NPedges} we have that $\vol_P \leq \alpha_1^2n^2 + 9n$. Similarly, $\vol_N \leq \alpha_2^2 n^2 + 9n$. By Lemma \ref{lowerBoundLemma}, we have 
\[
\lambda_1 \geq \frac{2}{|S| \sqrt{\vol_P \vol_N}} \geq (1+o(1))\frac{2}{\alpha_1\alpha_2\alpha_3 n^3}.
\]
Since $\alpha_1+\alpha_2 + \alpha_3 = 1$, this quantity is minimized when $\alpha_1 = \alpha_2= \alpha_3 = \frac{1}{3}$, and so 
\[
\lambda_1 \geq (1+o(1))\frac{54}{n^3}.
\]
\end{proof}

\section{Problems and remarks} \label{sec:conc}
In this paper, we proved an asymptotically sharp lower bound on the normalized Laplacian spectral gap of a connected graph. However, many  questions remain unanswered. Here we mention several related problems:

\begin{itemize}
\item Characterize the extremal graphs for which $\lambda_1 = \alpha(n)$. One might guess that all such extremal graphs are double kite graphs for large enough $n$, but we were not able to prove this.

\item Prove the corresponding theorem for the adjacency matrix: Stanic \cite{stanic2013graphs} conjectured that double kite graphs minimize the adjacency spectral gap.

\item Minimize $\lambda_1$ of the normalized Laplacian over the family of all regular graphs. Aldous and Fill \cite{aldous2002reversible} conjectured that the minimum is $(1+o(1))\frac{2\pi^2}{3n^2}$ and is achieved by a necklace graph. An affirmative answer to this conjecture was given for 3-regular graphs by \cite{brand2007characterization}, but the general case is still open. 

\end{itemize}

\bibliographystyle{siam}  %% This is just my personal favorite style. 
                              %There are many others.
 \bibliography{lambda1Refs}  %% This looks for the bibliography in myrefs.bib 
%                          which should be formatted as a bibtex file.

\end{document}